\documentclass{amsart}%
\usepackage{amsmath}
\usepackage{amsfonts,color}
\usepackage{amssymb}
\usepackage{graphicx}
\usepackage{amsfonts}%
\newtheorem{theorem}{Theorem}
\newtheorem*{acknowledgement}{Acknowledgement}

\newtheorem{corollary}[theorem]{Corollary}

\newtheorem{definition}[theorem]{Definition}

\newtheorem{lemma}[theorem]{Lemma}

\newtheorem{proposition}[theorem]{Proposition}
\newtheorem{remark}[theorem]{Remark}

\numberwithin{equation}{section}
\numberwithin{theorem}{section}
\begin{document}
\title[Constant scalar curvature metrics on $\partial C(n,4)$]{Constant Scalar Curvature Metrics on Boundary Complexes of Cyclic Polytopes}
\author{Daniel Champion}
\address[Daniel Champion]{University of Arizona, Tucson AZ, 85721}
\email{champion@math.arizona.edu}
\author{Andrew Marchese}
\address[Andrew Marchese]{SUNY Stoneybrook}
\email{yandyydna@gmail.com}
\author{Jacob Miller}
\address[Jacob Miller]{UCSD}
\email{jam003@ucsd.edu}
\author{Andrea Young}
\address[Andrea Young]{University of Arizona, Tucson AZ, 85721}
\email{ayoung@math.arizona.edu}
\begin{abstract}
In \cite{Gl}, a notion of constant scalar curvature metrics on piecewise flat manifolds is defined.  Such metrics are candidates for canonical metrics on discrete manifolds.  In this paper, we define a class of vertex transitive metrics on certain triangulations of $\mathbb{S}^3$; namely, the boundary complexes of cyclic polytopes.    We use combinatorial properties of cyclic polytopes to show that, for any number of vertices, these metrics have  constant scalar curvature.
\end{abstract}

\maketitle
\section{Introduction}
A classical question in differential geometry is that of finding canonical metrics on manifolds.   For example, the well-known Yamabe problem asks whether a given smooth compact Riemannian manifold $(M^n,g)$, for $n \geq 3$, has a constant scalar curvature metric $\tilde{g}$ which is conformal to $g$.    This problem was posed by Yamabe in 1960, and a complete solution to the problem is due to the combined work of Yamabe, Trudinger, Aubin, and Schoen.  See \cite{LP} for background on the problem and its solution.

Constant scalar curvature metrics arise as critical points in a conformal class of the volume normalized Einstein-Hilbert functional.  To see this, recall that the  Einstein-Hilbert functional is given by
\[\mathcal{EH}(g)=\int_{M^n}R_gdV_g,
\]
where $R_g$ is the scalar curvature of $g$ and $dV_g$ is the volume form of $g$.  Under a conformal variation $\delta g=fg$, the scalar curvature satisfies
\[\delta R[fg]=(1-n)\Delta f-Rf,
\]
and the volume form satisfies
\[\delta dV[fg]=\frac{n}{2}fdV.
\]
Thus the variation of $\mathcal{EH}$ is
\[\delta \mathcal{EH}[fg]=(\frac{n}{2}-1)\int_{M^n}RfdV.
\]
Critical points of this equation have $R=0$.  To obtain constant scalar curvature metrics as critical points, one can take the variation of the volume normalized $\mathcal{EH}$ functional:
\[\mathcal{NEH}(g)=\frac{\int_{M^n}R_gdV_g}{(\int_{M^n}dV_g)^{\frac{n-2}{n}}}.
\]

One might like to pose an analogue to the Yamabe problem in a non-smooth setting; for example, on piecewise flat manifolds.  To do so, one would need a discrete notion of curvature and of conformal variations.  In \cite{Gl}, such concepts are defined on piecewise flat manifolds.  The constant scalar curvature metrics in the piecewise flat setting should be somehow analogous to those in the smooth setting.  Therefore, these metrics are defined to be critical points of a discrete analogue of the Einstein-Hilbert functional defined by T. Regge in 1961 \cite{Re}.   We will call this functional the Einstein-Hilbert-Regge functional, and we will denote it $\mathcal{EHR}$.   See \cite{Ham} for a survey of the study of this functional as an action for general relativity and of its use in the Regge calculus and lattice gravity.   In \cite{Che}, it was shown that the $\mathcal{EHR}$ functional converges to the $\mathcal{EH}$ functional in a certain sense.  The idea of discretizing smooth geometric notions forms the basis for fields such as discrete differential geometry and discrete exterior calculus (\cite{BSu}, \cite{DHM}, \cite{DP}, \cite{MDSB}). 

In \cite{CGY},  the behavior of the $\mathcal{EHR}$ functional was analyzed on the double tetrahedron.  The double tetrahedron is the simplest triangulation of $\mathbb{S}^3$, and it is not simplicial.  However, the triangulation is neighborly and vertex transitive.   This note can be thought of as an extension of that work, as the piecewise flat manifolds we will consider are homeomorphic to $\mathbb{S}^3$ and are both neighborly and vertex transitive.

This paper is organized as follows.  In \S2, we will provide background about piecewise flat manifolds, and we will define discrete constant scalar curvature metrics.  In \S3, we will define cyclic polytopes.  The boundary complexes of the cyclic polytopes will be our main object of study, and we will recall some facts about these triangulations of $\mathbb{S}^3$.  We will define a certain class of metrics on the boundary complexes of cyclic polytopes known as cyclic length metrics in \S4, and we will collect some combinatorial results.  In \S5, we will show that  the metrics do indeed have constant scalar curvature.  Finally, in \S6, we mention some open questions and directions for future research.

\begin{acknowledgement}
We would like to thank the participants and organizers of the 2010 Arizona Summer Program, during which this work took place.  We would especially like to thank David Glickenstein for helpful discussions.
\end{acknowledgement}

\section{Background and Notation}
In this section, we will provide the necessary background on piecewise flat manifolds.  We will also describe a notion of conformal structure in this setting and will define constant scalar curvature metrics.  We will follow closely the definitions in \cite{CGY} and \cite{Gl}.

\subsection{Piecewise flat manifolds}

We begin with a \emph{triangulated piecewise
flat manifold}.  The dimension of a
triangulation is that of its highest dimensional simplex. A three-dimensional
triangulation $\mathcal{T}=\left(  V,E,F,T\right) $ has a collection of
vertices (denoted $V$), edges (denoted $E$), faces (denoted $F$), and
tetrahedra (denoted $T$).   In this paper, the triangulations we consider will be three-dimensional
simplicial complexes.  Thus we will require that every face borders exactly two tetrahedra and that the link of every vertex is homeomorphic to $\mathbb{S}^2$.

\begin{definition}
A triangulation is said to be \emph{neighborly} if, for any pair of vertices, the edge connecting them belongs to the simplicial complex.
\end{definition}

\begin{definition}
An \emph{automorphism} is a function $f: V\to V$ such that, for every edge $e=(u,v) \in E$, $f(e)=(f(u), f(v))$ is also an edge.  A triangulation is said to be \emph{vertex transitive} if the automorphism group acts transitively on the vertices.
\end{definition}

A triangulated piecewise flat manifold is denoted as
$\left(  M,\mathcal{T},\ell\right)  ,$ where $M$ is a manifold, $\mathcal{T}$
is a triangulation of $M,$ and $\ell$ is a metric according to the following definition.

\begin{definition}
\label{met def gen}A vector $\ell\in\mathbb{R}^{\left\vert E\right\vert }$
such that each simplex can be realized as a Euclidean simplex with edge
lengths determined by $\ell$ is called a \emph{metric} for the triangulated
manifold $\left(  M,\mathcal{T}\right)  ,$ and $\left(  M,\mathcal{T}%
,\ell\right)  $ is called a \emph{triangulated piecewise flat manifold}.
\end{definition}

The condition for a metric can be described using Cayley-Menger
determinants of the type described in \cite{CGY}.  Notice that a choice of metric (i.e. a choice of edge lengths) completely determines the geometry of the triangulation in the following sense.  The angles of a face are determined by the edge lengths via the
cosine law. The dihedral angles of a tetrahedron can then be calculated from
the angles at the faces using the spherical cosine law. We use $\beta_{e}$ to
refer to the dihedral angle of a tetrahedron at edge $e$.    If we wish to
emphasize that it is in tetrahedron $t,$ we denote it as $\beta_{e<t}.$  In what follows, $\tau<\sigma$ or $\sigma>\tau$ will mean that $\tau$ is a sub-simplex
of $\sigma.$  For computational purposes, we will sometimes use the notation $\beta_{ij,kl}$ to refer the dihedral angle at edge $e_{ij}$ in tetrahedron $t_{ijkl}$. 

We will also call a metric $\ell$ on $\left( M, \mathcal{T} \right)$ \emph{vertex transitive} if the automorphism group acts transitively on the vertices.  In this case, every vertex will ``look the same'' both combinatorially and geometrically.

\subsection{Discrete conformal structures and constant scalar curvature metrics}

In this section, we will consider a certain conformal structure that has been studied in \cite{CGY}, 
\cite{Gl}, \cite{Luo1}, \cite{RW}.  This choice of conformal structure will allow us to define a notion of vertex curvature and ultimately of constant scalar curvature metrics.   
\begin{definition}
\label{conf class def}
Let $\{L_{e}\}_{e\in E}$ be such that $(M,\mathcal{T},L)$ is a piecewise flat
manifold. Let $V^{\ast}$ denote
the real-valued functions on the vertices, and let $U\subset V^{\ast}$ be an open set. A \emph{conformal structure}
is a map $U\rightarrow\mathfrak{met}\left(  M,\mathcal{T}\right)  $ determined
by
\begin{equation}
\ell_{e}\left(  f\right)  =\exp\left[  \frac{1}{2}\left(  f_{v}+f_{v^{\prime}%
}\right)  \right]  L_{e}, \label{pb conf}%
\end{equation}
where $e$ is the edge between $v$ and $v^{\prime}.$ The \emph{conformal class}
is the image of $U$ in $\mathfrak{met}\left(  M,\mathcal{T}\right)  ,$ and it
is entirely determined by $L.$ A \emph{conformal variation} $f\left(
t\right)  $ is a smooth curve $\left(  -\varepsilon,\varepsilon\right)
\rightarrow U$ for small $\varepsilon>0$, and it induces a \emph{conformal
variation of metrics} $\ell\left(  f\left(  t\right)  \right)  .$
\end{definition}

\begin{remark}
There is a more general notion of conformal structure on piecewise flat
manifolds that is described in \cite{Gl}. The conformal structure described
here is called the perpendicular bisector conformal structure in that paper.
\end{remark}

We define both the edge and vertex curvatures.   The edge curvature is independent of a conformal structure.  However,  as described in \cite{Gl}, the choice of conformal structure allows us to define the vertex curvature.  This will lead to a
definition of constant scalar curvature metrics.

\begin{definition}
The \emph{edge curvature} of an edge $e$ is
\begin{equation}
\label{edge curvature}
K_e=(2\pi-\sum_{t\in T}\beta_{e<t})\ell_e,
\end{equation}
where $\ell_e$ is the edge length of $e$.

The \emph{vertex curvature} of a vertex $v$ is%
\begin{equation}
\label{vertex curvature}
K_{v}=\frac{1}{2}\sum_{e>v}K_e,%
\end{equation}
where the sum is over all edges containing vertex $v$.
\end{definition}

We now define a discrete analogue of the Einstein-Hilbert functional.

\begin{definition}The \emph{Einstein-Hilbert-Regge functional} is
\begin{equation}
\mathcal{EHR}\left(  M,\mathcal{T},\ell\right)  =\sum_{v\in V}K_{v}.
\end{equation}
\end{definition}

Our goal is to define constant scalar curvature metrics in a way that is analogous to the smooth setting; namely, we would like to let these metrics be critical points of the normalized $\mathcal{EHR}$ functional.   In the smooth case, one usually considers a volume normalization.  However, in the discrete case, since the formula for volume of a simplex is quite complicated, one may also consider a normalization which is linear in the edge lengths.  Thus, we will  consider the two normalizations of the $\mathcal{EHR}$ functional given in \cite{CGY}.

To define these normalizations, we need the definitions of the total length and the total volume of the triangulation.

\begin{definition}
The \emph{total length} of $\left(  M,\mathcal{T},\ell\right)  $ is
\begin{equation}
\mathcal{L}\left(  M,\mathcal{T},\ell\right)  =\sum_{e\in E}\ell_{e}.
\end{equation}
Let $V_{t}$ be the volume of tetrahedron $t.$ Then the \emph{volume} of
$\left(  M,\mathcal{T},\ell\right)  $ is
\begin{equation}
\mathcal{V}\left(  M,\mathcal{T},\ell\right)  =\sum_{t\in T}V_{t}.
\end{equation}
\end{definition}

\begin{definition}
The \emph{length normalized Einstein-Hilbert-Regge functional} is
\begin{equation}
\mathcal{LEHR}\left(  M,\mathcal{T},\ell\right)  =\frac{\mathcal{EHR}\left(
M,\mathcal{T},\ell\right)  }{\mathcal{L}\left(  M,\mathcal{T},\ell\right)  }.
\label{LEHR}%
\end{equation}
The \emph{volume normalized Einstein-Hilbert-Regge functional} is
\begin{equation}
\mathcal{VEHR}\left(  M,\mathcal{T},\ell\right)  =\frac{\mathcal{EHR}\left(
M,\mathcal{T},\ell\right)  }{\mathcal{V}\left(  M,\mathcal{T},\ell\right)
^{1/3}}. \label{VEHR}%
\end{equation}

\end{definition}

The normalizations are defined so that the functionals take the same value if
all lengths are scaled by the same positive constant.

Given these definitions of normalized $\mathcal{EHR}$ functionals, we can compute their derivatives with respect to a conformal variation.   We define constant scalar curvature metrics to be critical points in a conformal class of the normalized $\mathcal{EHR}$ functionals.  The subsequent definitions follow immediately from the variation formulas of $\mathcal{LEHR}$ and $\mathcal{VEHR}$, which can be found in \cite{CGY}, Lemma 4.9.
\begin{definition}
\label{csc def}A three-dimensional piecewise flat manifold $(M,\mathcal{T}%
,\ell)$ has \emph{constant scalar curvature} if it is a critical point of one
of the normalized $\mathcal{EHR}$ functionals with respect to a conformal
variation. 

We say $(M,\mathcal{T},\ell)$ is $\mathcal{LCSC}$ if
\begin{equation}
K_{v}=\lambda_{\mathcal{L}}L_{v}, \label{Lcsc eqn}%
\end{equation}
for all $v\in V$. Here $\lambda_{\mathcal{L}}=\mathcal{LEHR}$, and $L_{v}=\frac{1}{2}\sum_{e>v}\ell_{e}$. 

We say
$(M,\mathcal{T},\ell)$ is $\mathcal{VCSC}$ if
\begin{equation}
K_{v}=\lambda_{\mathcal{V}}V_{v}, \label{Vcsc eqn}%
\end{equation}
for all $v\in V$. Here $\lambda_{\mathcal{V}}=\frac{\mathcal{EHR}%
}{3\mathcal{V}}$, $V_{v}=\frac{1}{3}\sum
_{t>f>v}h_{f<t}A_{f}$, $h_{f<t}$ is the signed distance between the circumcenters of the face $f$ and tetrahedron $t$, and $A_{f}$ is the area of face $f$.
\end{definition}

\begin{remark}
Note that
$\sum_{v\in V}L_{v}=\mathcal{L}$, and $\sum_{v\in V}V_{v}=3\mathcal{V}.$
\end{remark}

\section{Cyclic Polytopes}
Our primary objects of interest in the remainder of the paper will be certain triangulations of $\mathbb{S}^3$.  Specifically, we will study the piecewise flat 3-manifolds that are the boundary complexes of 4-dimensional cyclic polytopes.  These classical combinatorial manifolds are well-known and provide examples of neighborly triangulations of $\mathbb{S}^3$.  For the convenience of the reader, we define cyclic polytopes as in \cite{Gr}, and we recall some important facts.

\begin{definition}
The \emph{moment curve} is a map $M_4: \mathbb{R}\to \mathbb{R}^4$ defined parametrically as
\begin{equation}
\label{moment curve}
x(t)=(t,t^2,t^3,t^4).
\end{equation}
\end{definition}

\begin{definition}
A \emph{4-dimensional cyclic polytope}, $C(n,4)$, is the convex hull of $n\geq 5$ points $x(t_i)$ on $M_4$, where $t_1<t_2<\cdots<t_n$.
\end{definition}

We will define the set $V$ as follows:
\begin{equation}
\label{def of V}
V=\{x(t_i)\in \mathbb{R}^4 : 1 \leq i \leq n \}.
\end{equation}
Here we will denote elements of $V$ as $v_i$ for $1\leq i \leq n$, where $v_i=x(t_i)$.

 The following condition on the tetrahedra in 4-dimensional cyclic polytopes is due to Gale (see \cite{Gr}).

\begin{theorem}[Gale's Evenness Condition]
\label{Gale's}
Consider a cyclic 4-polytope $C(n,4)$. A set of four points $\tilde{V}\subseteq V$ determines a tetrahedron in $C(n,4)$ if and only if every two points in $V\setminus \tilde{V}$ are separated on $M_4$ by an even number of points of $\tilde{V}$.
\end{theorem}

It is well-known that for any $n\geq 5$, the boundary complex of $C(n,4)$ forms a neighborly triangulation of $\mathbb{S}^3$.  This triangulation is invariant under the vertex transitive action of the dihedral group $D_n$ \cite{KL}.
The triangulation has $n$ vertices (given by $V$ in Eq.~\ref{def of V}), $\frac{n(n-1)}{2}$ edges, $n^2-3n$ faces, and $\frac{n^2-3n}{2}$ tetrahedra.

\begin{definition}
We will let $\partial C(n,4)$ denote the triangulation of $\mathbb{S}^3$ given by the boundary complex of $C(n,4)$.
\end{definition}

We define a distinguished cycle, $\mathcal{C}$, to be the following set of edges:
\begin{equation}
\label{def of C}
\mathcal{C}=\{e\in E: e=v_iv_{i+1} \mbox{ for }1\leq i <n \mbox{ or } e=v_nv_1 \}.
\end{equation}
A straightforward application of Theorem~\ref{Gale's} allows us to ascertain when four points in $V$ determine a tetrahedron in $\partial C(n,4)$.  

\begin{corollary}
\label{edgecor}
A set of four points in  $\partial C(n,4)$  forms a tetrahedron if and only if those four points yield two distinct non-local edges in $\mathcal{C}$.
\end{corollary}

The simplest example of a cyclic $4$-polytope is $C(5,4)$, which is the $4$-simplex.  In this case, $\partial C(5,4)$ is known as the pentachoron, which is the smallest simplicial triangulation of $\mathbb{S}^3$.  See \cite{Ch} for many geometric results about this manifold.

\section{Cyclic length metrics on $\partial C(n,4)$}
In this section, we  define certain metrics on  $\partial C(n,4)$, known as cyclic length metrics.  We will collect some useful facts about these metrics.   In the next section, we will use these facts to show that these metrics have constant scalar curvature.

\begin{definition}
Let $\ell$ be a metric on $\partial C(n,4)$.  Then $\ell$ is called a \emph{cyclic length metric} if the length of any edge $e_{ij}$ is a function of the minimum number of edges between $v_i$ and $v_j$ on $\mathcal{C}$, where $\mathcal{C}$ is the distinguished cycle defined by Eq.~\ref{def of C}.  
\end{definition}
\noindent
See Figure~\ref{cyclic metric pic} for a representation of a cyclic length metric on $\partial C(11,4)$.

We will denote the minimal number of edges between $v_i$ and $v_j$ on $\mathcal{C}$ as $D_{ij}$.  Note also that although the lengths of each edge $e_{ij}$ are a function of  $D_{ij}$, these lengths must still satisfy the criteria required of a metric; i.e. the Cayley-Menger determinant must be positive for all tetrahedra.

\begin{remark}
\label{cl is vt}
By construction, a cyclic length metric $\ell$ on $\partial C(n,4)$ is vertex transitive.
\end{remark}

On the pentachoron, $\partial C(5,4)$, any vertex transitive metric is a cyclic length metric.  However, in general there may be vertex transitive metrics on $\partial C(n,4)$ that are not cyclic length metrics.

When considering  cyclic length metrics on $\partial C(n,4)$, we require additional information about the combinatorics and symmetries of these manifolds. The analysis will vary slightly depending on whether $n$ is odd or even; for reasons to be made clear below, we let odd $n=2m+3$ and even $n=2m+2$.  We summarize the results of this section in the following table:

\begin{table}[h]
\caption{Properties of cyclic length metrics on $\partial C(n,4)$}
    \begin{tabular}{lll}
    Quantity of interest \hspace{.7in} & $n=2m+3$ \hspace{.7in} & $n=2m+2$\\
    \hline \\
    Number of distinct edge lengths &  $m+1$ & $m+1$\\
    &&\\
    Number of tetrahedra &  $2m^2+3m$ & $2m^2+m-1$\\
    &&\\
   Types of tetrahedra &  $m$  & $m$ \\
   &&\\
   Number of each type & $2m+3$ & \parbox{3.5cm}{$2m+2$\\ ( $m+1$ of type $T^{e}_m$)}\\
    \end{tabular}
\end{table}

\begin{center}
\begin{figure}
\includegraphics[height=60mm]{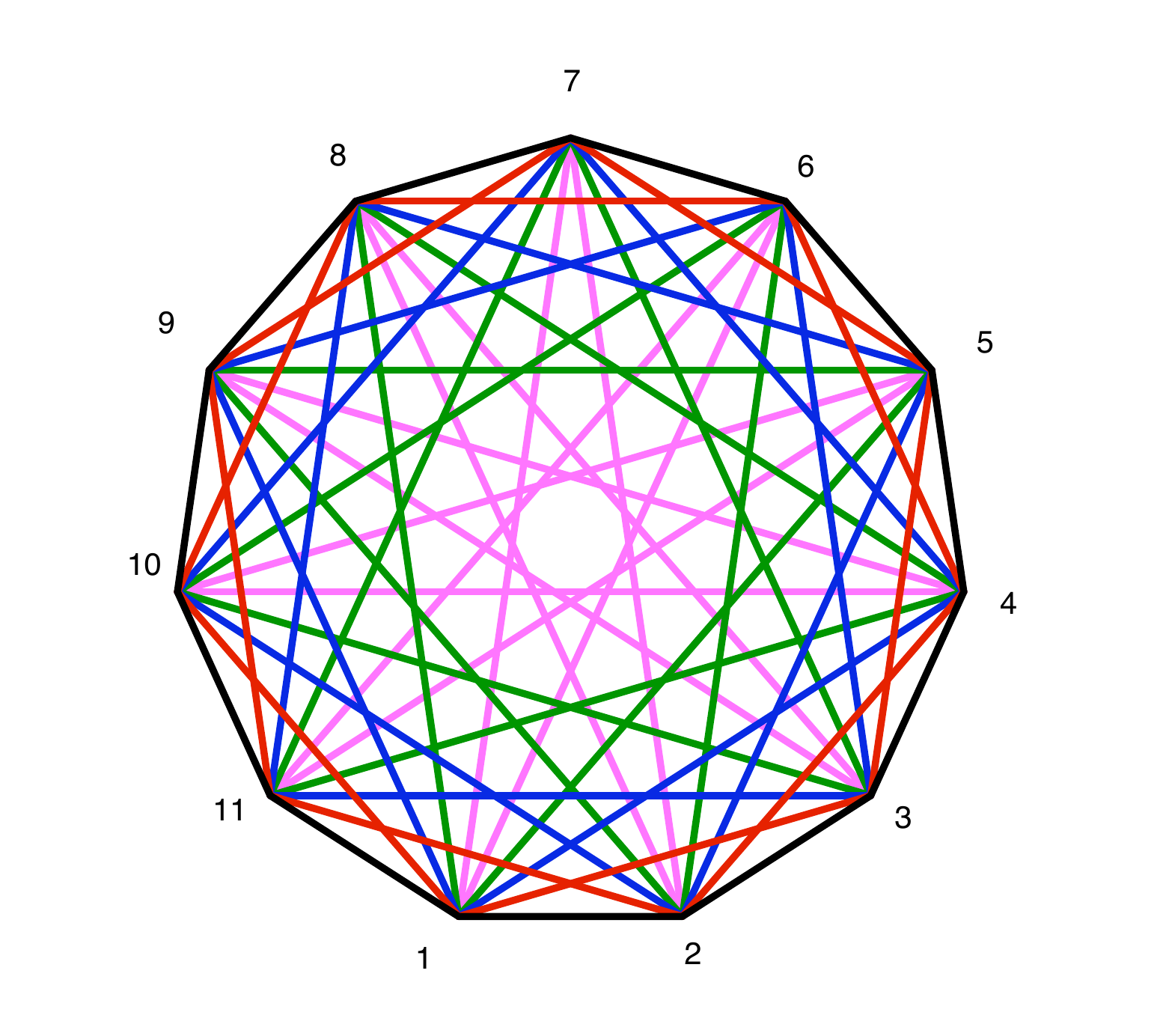}
\caption{A two-dimensional representation of  $\partial C(11,4)$ with a cyclic length metric}
\label{cyclic metric pic}
\end{figure}
\end{center}

\begin{lemma}
	\label{com}
Let  $\ell$ be a cyclic length metric on $\partial C(n,4)$.  
\begin{enumerate} 
\item Let $m \in \mathbb{N}$ be such that $n=2m+3$.  The triangulation has the following properties:
\begin{enumerate}
\item  There can be up to a total of $m+1$ distinct edge lengths.
\item There are $2m^2+3m$ total tetrahedra in the triangulation.
\end{enumerate}
\medskip
\item Let $m\in \mathbb{N}$ be such that $n=2m+2$.  The triangulation has the following properties:
 \begin{enumerate}
\item  There can be up to a total of m+1 distinct edge lengths.
\item There are $2m^2+m-1$ total tetrahedra in the triangulation.
\end{enumerate}
\end{enumerate}
\end{lemma}
\begin{proof}
We begin with the case of $n=2m+3$.  To prove the first statement, we choose two arbitrary distinct $v_i, v_j\in$V. Since there are $2m+3$ vertices, there are $2m+3$ edges in the distinguished cycle $\mathcal{C}$ (see Eq.~\ref{def of C}).    Recall that the edge lengths in a cyclic length metric are determined by the minimum number of edges between $v_i$ and $v_j$ on $\mathcal{C}$.   Notice that the minimum number of edges in $\mathcal{C}$ separating $v_i$ and $v_j$ is at most $\lfloor \frac{2m+3}{2} \rfloor$.  This is $m+1$, as claimed.

For the second statement, choose an arbitrary edge in $\mathcal{C}$ and a second edge in $\mathcal{C}$ not local to the first edge.  By Corollary~\ref{edgecor}, all such combinations will generate the tetrahedra in the triangulation. The number of these combinations, and thus the number of  tetrahedra, is exactly $$\frac{(2m+3)((2m+3)-3)}{2} = 2m^2+3m.$$

In the case that $n=2m+2$, the analysis follows in the same fashion when one recalls there are $2m+2$ edges in the distinguished cycle $\mathcal{C}$.
\end{proof}

\begin{corollary}
\label{path length}
Let  $\ell$ be a cyclic length metric on $\partial C(n,4)$.    Recall that the minimum number of edges between $v_i$ and $v_j$ on $\mathcal{C}$ is denoted $D_{ij}$.
\begin{enumerate} 
\item Let $m \in \mathbb{N}$ be such that $n=2m+3$. If there exists a path of $p$ distinct edges between $v_i$ and $v_j$ on $\mathcal{C}$ and $p\leq m+1$, then $D_{ij}=p$. If $m+1<p$, then $D_{ij}=(2m+3) - p$.
\item Let $m \in \mathbb{N}$ be such that $n=2m+2$.  If there exists a path of $p$ distinct edges between $v_i$ and $v_j$ on $\mathcal{C}$ and $p\leq m+1$,  then $D_{ij}=p$. If $m+1<p$, then $D_{ij}=(2m+2) - p$.
\end{enumerate}
\end{corollary}
\begin{proof}
These facts follow directly from Lemma \ref{com}.
\end{proof}

Recall that Euclidean tetrahedra are completely determined by their edge lengths.  The following lemma gives an upper bound for the different length structures that appear on tetrahedra in $(\partial C(n,4), \ell)$.
\begin{lemma}
\label{tetralemma}
Let $\ell$ a cyclic length metric on $\partial C(n,4)$.  
\begin{enumerate}
\item  Let $n=2m+3$.  There are at most $m$ distinct types of tetrahedra in the triangulation, denoted type $T_k$ for $1\leq k < m$ and $T_m^o$ for $k=m$.   For each $k$, there are 2m+3 tetrahedra of  type $T_k$ in the triangulation.
\medskip
\item Let $n=2m+2$.  There are $m$ distinct types of tetrahedra in the triangulation, denoted type $T_k$ for $1\leq k < m$ and $T_m^e$ for $k=m$.  For $1\leq k <m$, there are $2m+2$ type $T_k$ tetrahedra in the triangulation, whereas there are $m+1$ type $T^{e}_m$ tetrahedra.
\end{enumerate}
\end{lemma}
\begin{proof}
To start, let $n=2m+3$.  By the proof of Lemma \ref{com}.1, there are at most $m+1$ edge lengths in the triangulation. Let the set of these edge lengths be given by $\{L_1, L_2,...,L_{m+1}\}$, where $L_k$ corresponds to the case that $D_{ij}=k$.  Clearly the maximum number of tetrahedra is generated when all of the $L_i$ are distinct, so we will consider that case.  Choose two non-local edges in $\mathcal{C}$ that have vertices $v_i, v_{i+1}, v_j, v_{j+1}$.   By Corollary~\ref{edgecor}, these four vertices generate a tetrahedron in $\partial C(2m+3,4)$.  We claim that this tetrahedron is determined by the distance between $v_{i+1}$ and $v_{j}$ on $\mathcal{C}$, i.e., by $D_{(i+1)j}$.  

We first consider the case that $D_{(i+1)j}<m$; say, $D_{(i+1)j}=p$ for $p<m$.  Notice that the tetrahedron generated by the vertices $v_i, v_{i+1}, v_j, v_{j+1}$ will have the following edge lengths:
$\ell(e_{i(i+1)})=\ell(e_{j(j+1)})=L_1$, $\ell(e_{(i+1)j})=L_p$, $ \ell(e_{ij})=\ell(e_{(i+1)(j+1)})=L_{p+1}$, $\ell(e_{i(j+1)})=L_{p+2}.$   See Figure~\ref{T_p tetra}.  Since the $L_p$ are distinct, this gives $m-1$ distinct tetrahedra in $\partial C(2m+3,4)$.   We will call each of these types of tetrahedra \emph{type $T_p$}.
\begin{center}  
\begin{figure}
\includegraphics[height=60mm]{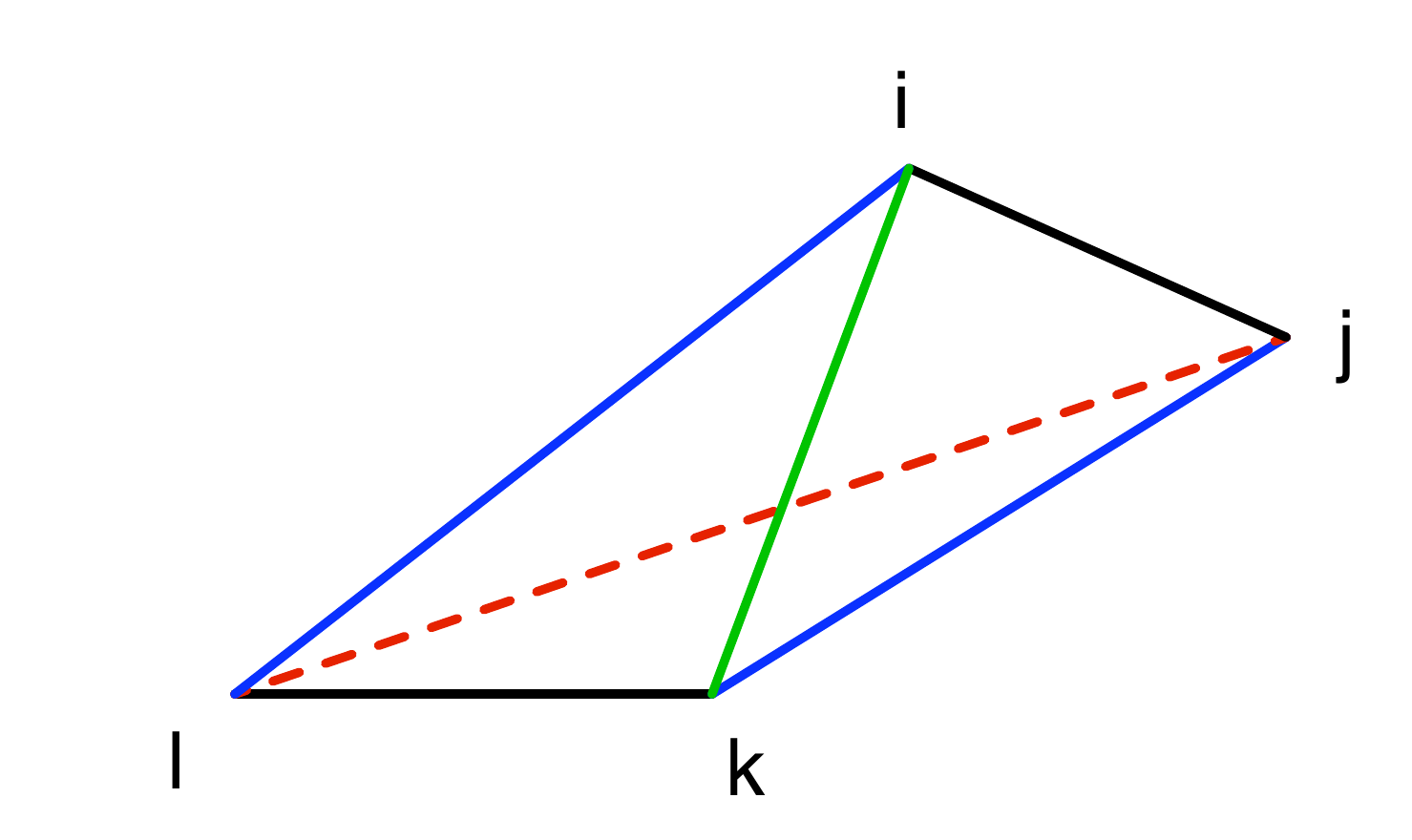}
\caption{Type $T_p$ tetrahedron: $\ell_{ij}=\ell_{kl}=L_1$, $\ell_{ik}=L_p$, $\ell_{il}=\ell_{jk}=L_{p+1}$, $\ell_{jl}=L_{p+2}$}
\label{T_p tetra}
\end{figure}
\end{center}

The second case is that $D_{(i+1)j}=m$.  In this case, the tetrahedron generated by the vertices $v_i, v_{i+1}, v_j, v_{j+1}$ will have the following edge lengths:
$\ell(e_{i(i+1)})=\ell(e_{j(j+1)})=L_1$, $\ell(e_{(i+1)j})=L_m$, $ \ell(e_{ij})=\ell(e_{(i+1)(j+1)})=\ell(e_{i(j+1)})=L_{m+1}.$   See Figure~\ref{T_m^o tetra}.  There is clearly one type of tetrahedra of this form in $\partial C(2m+3,4)$.  We will call this type of tetrahedron \emph{type $T^{o}_m$}.

\begin{center}
\begin{figure}
\includegraphics[height=60mm]{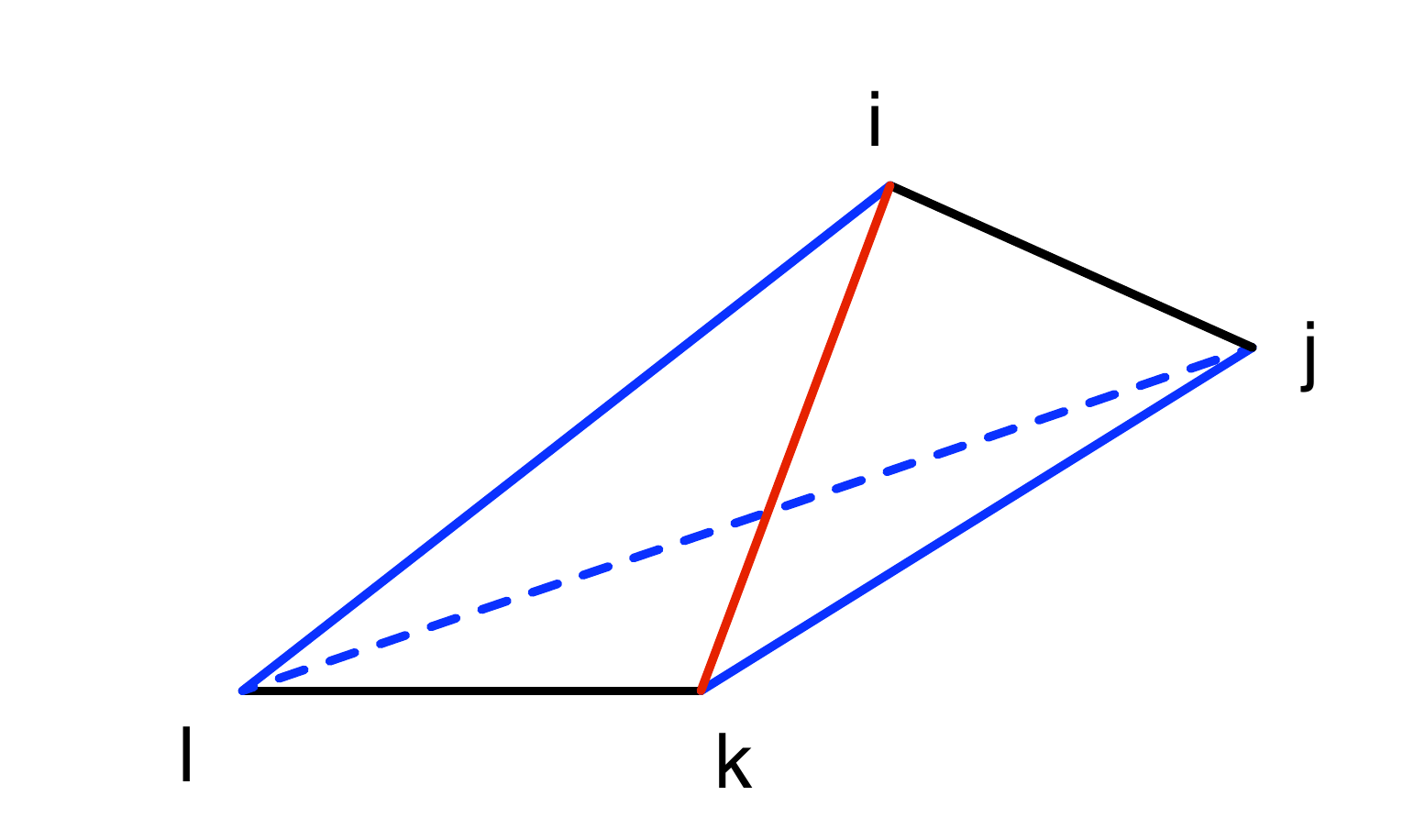}
\caption{Type $T^{o}_m$ tetrahedron: $\ell_{ij}=\ell_{kl}=L_1$, $\ell_{ik}=L_m$, $\ell_{il}=\ell_{jk}=\ell_{jl}=L_{m+1}$}
\label{T_m^o tetra}
\end{figure}
\end{center}

It is also possible that $D_{(i+1)j}>m$; however, by Corollary~\ref{path length}, we must then have $D_{i(j+1)}<m$, so we do not generate any additional types of tetrahedra.

Finally, we show how many of each type of  tetrahedra exist in our triangulation when $n=2m+3$.  There are  $2m+3$ edges on $\mathcal{C}$.  Notice that exactly two tetrahedra of each type will be formed with each of these edges.  However, in this process we double count the tetrahedra, so the total number of tetrahedra of a given type is $2m+3$.

\bigskip
Now consider the case that $n=2m+2$.  By Lemma~\ref{com}.2, there are at most $m+1$ edge lengths in the triangulation. As above, let the set of all these edge lengths be given by $\{L_1, L_2,...,L_{m+1}\}$, where $L_k$ corresponds to the case that $D_{ij}=k$, and all of the $L_k$ are distinct.  Choose two non-local edges in $\mathcal{C}$ having vertices $v_i, v_{i+1}, v_j, v_{j+1}$.   These four vertices generate a tetrahedron in $\partial C(2m+2,4)$, and we claim that this tetrahedron is determined by $D_{(i+1)j}$.

The case that $D_{(i+1)j}<m$ follows in exactly the same fashion as the corresponding case in the proof for $n=2m+3$.  We obtain $m-1$ distinct types of tetrahedra with edge lengths as in Figure~\ref{T_p tetra}, and we call each of these types of tetrahedra type $T_p$.

In the case that $D_{(i+1)j}=m$, we have a difference in the edge lengths as compared to the odd vertex case.  The key fact is that, in this setting, $D_{i(j+1)}=m$.  Thus, the tetrahedron generated by $v_i, v_{i+1}, v_j, v_{j+1}$ will have the following edge lengths:  $\ell(e_{i(i+1)})=\ell(e_{j(j+1)})=L_1$, $\ell(e_{(i+1)j})=\ell(e_{i(j+1)})=L_m$, $ \ell(e_{ij})=\ell(e_{(i+1)(j+1)})=L_{m+1}.$   See Figure~\ref{T_m^e tetra}.  We call this a type $T^{e}_m$ tetrahedron.

\begin{center}
\begin{figure}
\includegraphics[height=50mm]{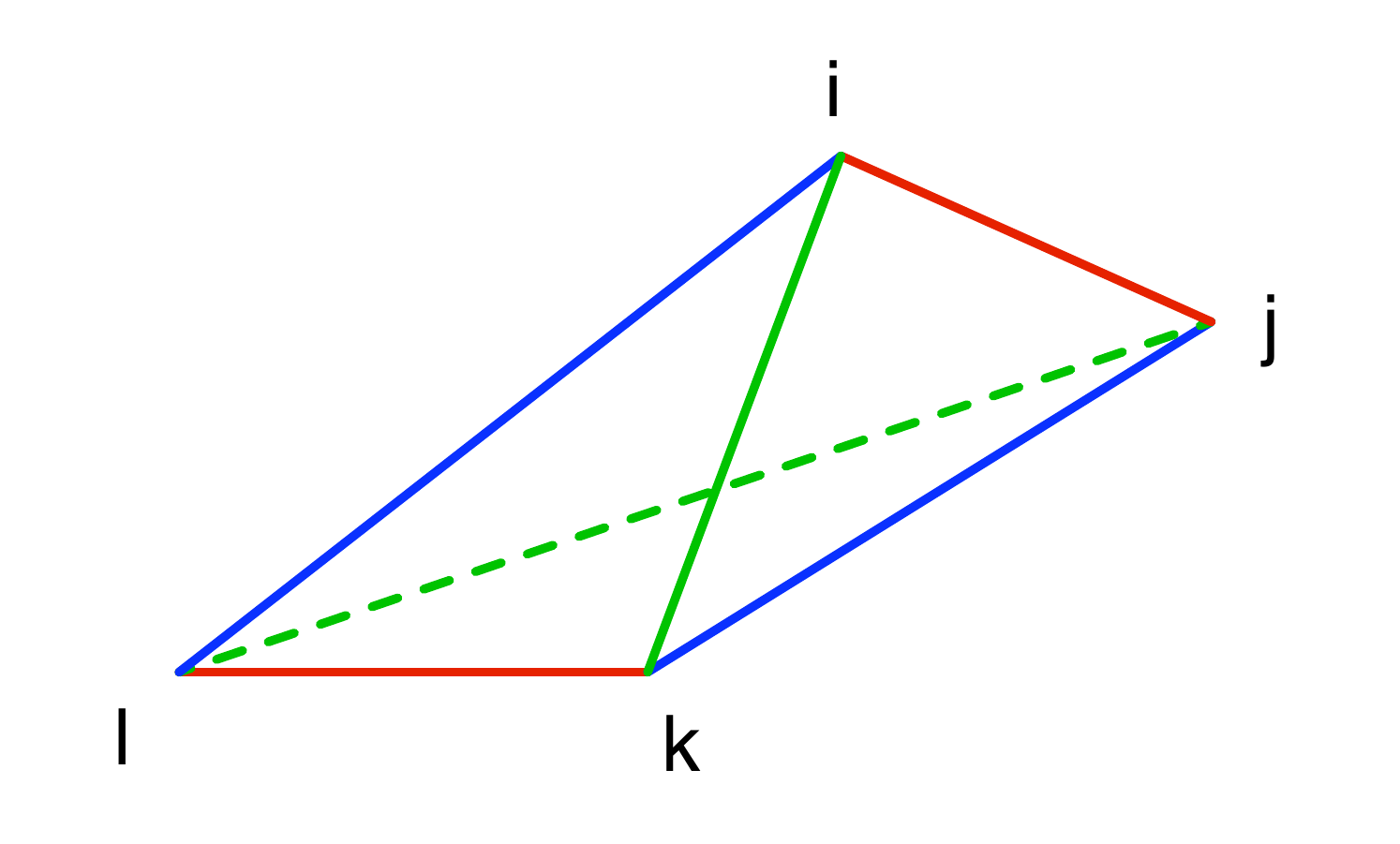}
\caption{Type $T^{e}_m$ tetrahedron: $\ell_{ij}=\ell_{kl}=L_1$, $\ell_{ik}=\ell_{jl}=L_m$, $\ell_{il}=\ell_{jk}=L_{m+1}$}
\label{T_m^e tetra}
\end{figure}
\end{center}
As we saw previously, if $D_{(i+1)j}>m$, no new tetrahedra are generated.
\bigskip

Finally, we show  how many of each type of tetrahedra are in our triangulation when $n=2m+2$.  There are $2m+2$ edges on $\mathcal{C}$.  Exactly two tetrahedra of each type $T_k$ for $1\leq k <m$ will be formed with each of these edges. Further, each edge will form exactly one type $T_m^e$ tetrahedron. In this process, we double count the tetrahedra.  So for $1\leq k <m$, the total number of type $T_k$ is $\frac{2(2m+2)}{2}=2m+2$. Similarly, the total number of type $T^e_m$ is $\frac{2m+2}{2}=m+1$.

\end{proof}

Notice that the tetrahedra of type $T_m^e$ have all pairs of opposite edges equal.  These tetrahedra are known as ``equihedral'', ``isosceles'', or ``equifacial'' tetrahedra in the literature, and they are very well-studied.  Indeed, there are over 100 equivalent conditions that characterize the equihedral tetrahedra.  They also give rise to constant scalar curvature metrics  
on the double tetrahedron (see \cite{CGY}).  One of the equivalent conditions that determines an equihedral tetrahedron  is that opposite dihedral angles are equal.  The following proposition shows an analogue of that condition for tetrahedra of type $T_k$.

\begin{proposition}
\label{eq length eq dih}
Let $\ell$ be a cyclic length metric on $\partial C(n,4)$.  In a tetrahedron of type $T_k$ for $1\leq k \leq m$, edges with equal lengths have the same dihedral angles.
\end{proposition}
\begin{proof}
Consider a type $T_k$ tetrahedron for some $k \in [1, m]$.  By Lemma~\ref{tetralemma}, there are two pairs of opposite edges that have equal lengths and one pair that may or may not.  (If all of the edge lengths are distinct, then this pair will not have equal lengths, except in the case of $T_m^e$.)  Let the vertices of the tetrahedron be labeled as in Figure~\ref{T_p tetra}.   Let $\ell_{ij}=\ell_{kl}=a$, $\ell_{ik}=b$, $\ell_{il}=\ell_{jk}=c$, and $\ell_{jl}=d$.  We will show that $\beta_{il}=\beta_{jk}$.  The argument that $\beta_{ij}=\beta_{kl}$ follows in the same way.

One can use the spherical law of cosines to compute $\beta_{il}$ to be
\begin{equation}
\label{beta il}
\beta_{il}=\cos ^{-1}\left(\frac{\frac{a^2+b^2-c^2}{2 a
   b}-\frac{\left(-a^2+b^2+c^2\right) \left(a^2+c^2-d^2\right)}{4 a b
   c^2}}{\sqrt{1-\frac{\left(-a^2+b^2+c^2\right)^2}{4 b^2 c^2}}
   \sqrt{1-\frac{\left(a^2+c^2-d^2\right)^2}{4 a^2 c^2}}}\right).
\end{equation}

Similarly, 
\begin{equation}
\label{beta jk}
\beta_{jk}=\cos ^{-1}\left(\frac{\frac{a^2-c^2+d^2}{2 a
   d}-\frac{\left(a^2-b^2+c^2\right) \left(-a^2+c^2+d^2\right)}{4 a c^2
   d}}{\sqrt{1-\frac{\left(a^2-b^2+c^2\right)^2}{4 a^2 c^2}}
   \sqrt{1-\frac{\left(-a^2+c^2+d^2\right)^2}{4 c^2 b^2}}}\right).
\end{equation}
Algebraic manipulation shows that Eq.~\ref{beta il} becomes
\begin{equation}
\beta_{il}=\cos ^{-1}\left(\frac{a^4-\left(b^2-2 c^2+d^2\right) a^2-3 c^4+c^2
   d^2+b^2 \left(c^2+d^2\right)}{4 a b c^2
   \sqrt{\frac{\left(a^4-2 \left(b^2+c^2\right) a^2+\left(b^2-c^2\right)^2\right)
   \left(a^4-2 \left(c^2+d^2\right) a^2+\left(c^2-d^2\right)^2\right)}{16
   a^2 b^2c^4 }}}\right),
\end{equation}
while Eq.~\ref{beta jk} becomes
\begin{equation}
\beta_{jk}=\cos ^{-1}\left(\frac{a^4-\left(b^2-2 c^2+d^2\right) a^2-3 c^4+c^2
   d^2+b^2 \left(c^2+d^2\right)}{4 a c^2d
   \sqrt{\frac{\left(a^4-2 \left(b^2+c^2\right) a^2+\left(b^2-c^2\right)^2\right)
   \left(a^4-2 \left(c^2+d^2\right) a^2+\left(c^2-d^2\right)^2\right)}{16
   a^2 c^4 d^2}}}\right).
\end{equation}
Thus we see that $\beta_{il}=\beta_{jk}$.
\end{proof}

\begin{lemma}
\label{fourdis}
Let $\ell$ be a cyclic length metric on $\partial C(n,4)$.   
\begin{enumerate}
\item  Let $n=2m+3$.  Any vertex $v_i \in V$ is local to four tetrahedra of the types $T_k$, with $1 \leq k<m$, and to four tetrahedra of the type $T^o_m$.
\medskip
\item  Let $n=2m+2$.   Any vertex $v \in V$ is local to  four tetrahedra of the type $T_k$, with $1\leq k<m$, and to two tetrahedra of the type $T^e_m$.
\end{enumerate}
\end{lemma}
\begin{proof}
We begin with the case that $n=2m+3$, and we choose a vertex $v_0 \in V$. It is now possible to construct a tetrahedron of type $T_k$ with vertices $v_0$, $v_1$, $v_{k+1}$, $v_{k+2}$.  (Notice that this construction is valid for all $1 \leq k \leq m$; in particular, for type $T_m^o$.)   Let this specific tetrahedron be labeled $t_{k1}$.

Define a map $f:V\rightarrow V$ by
\begin{equation}
\label{def of fn}
f(v_i) = v_{i+h (\mbox{mod } 2m+3)},
\end{equation}
where h solves the equation $0 \equiv k+1+h (\mbox{mod }  2m+3)$.
Note that under this map, $v_0 \to v_h$, $v_1\to v_{h+1}$, $v_{k+1}\to v_0$, and $v_{k+2}\to v_1$. Since this map preserves $D_{ij}$ for any two vertices $v_i$ and $v_j$, this is still a type $T_k$ tetrahedron. Also note that this tetrahedron is distinct from $t_{k1}$; i.e. that the new set of vertices $v_h, v_{h+1}, v_0, v_1$ is distinct from our original set of vertices.   If not, then we would have $h \equiv k+1 (\mbox{mod  }2m+3)$, which would imply that $h\equiv 0 (\mbox{mod }2m+3)$, which contradicts our choice of $h$.  Let this tetrahedron be labeled $t_{k2}$.

Define the tetrahedron $t_{k3}$ by vertices $v_{2m+2}, v_0, v_k, v_{k+1}$. This is distinct from $t_{k1}$ and $t_{k2}$ and is clearly a tetrahedron of type $T_k$.  Now apply $f$ to $t_{k3}$. Under this map, $v_{2m+2}\to v_{h-1}$, $v_0\to v_h$, $v_{k+1}\to v_0$, and $v_{k}\to v_{2m+2}$. Label this tetrahedron $t_{k4}$.

Again, we note that $t_{k4}$ is distinct from $t_{k3}$.  If not, then we would have $h-1 \equiv k (\mbox{mod }2m+3)$, which gives a similar contradiction as above.  Also note that $t_{k4}$ is distinct from $t_{k1}$ since $k+1 \neq k+2$, and it is distinct from $t_{k2}$ since $h \neq h+1$.  As our choice of $v_0$ was arbitrary, every vertex is part of at least four distinct tetrahedra of type $T_k$, with $1\leq k \leq m$.

We claim that in fact each vertex is contained in exactly four distinct tetrahedra of type $T_k$.  Let $v_{val}$ denote the number of tetrahedra incident to vertex $v$.  From above, we know that $v_{val}\geq 4m$.  In fact, since our triangulation is vertex transitive, $v_{val}=\frac{4T}{V}$, where $T$ is the total number of tetrahedra in the triangulation and $V$ is the total number of vertices.   Thus $v_{val}=\frac{4(2m^3+3m)}{2m+3}=4m$, so each vertex is contained in exactly four distinct tetrahedra of type $T_k$ for $1 \leq k \leq m$ (including type $T_k^o$) as required.

\bigskip

The proof in the case that $n=2m+2$ follows in much the same way for type $T_k$ tetrahedra when  $1\leq k < m$.  One can simply change the definition of the function $f$ in Eq.~\ref{def of fn} to be $f(v_i)=v_{i+h (\mbox{mod } 2m+2)}.$  However, to construct the two tetrahedra of type $T_m^e$, we choose the vertex sets $v_0, v_1, v_{m+1},v_{m+2}$ and $v_{2m+2}, v_0, v_m, v_{m+1}$.  Using a similar counting argument as above, we see that each vertex lies in four type $T_k$ tetrahedra for $1\leq k <m$ and in two type $T^e_m$ tetrahedra.
\end{proof}

\section{Cyclic Length Metrics have Constant Scalar Curvature}
In this section, we will use the results from the previous section to show that cyclic length metrics on $\partial C(n,4)$ have constant scalar curvature; namely, that they solve both Eq.~\ref{Lcsc eqn} and Eq.~\ref{Vcsc eqn}.

We first show that both the edge and vertex curvatures (Eqs.~\ref{edge curvature} and \ref{vertex curvature}) are constant for cyclic length metrics on $\partial C(n,4)$.

\begin{lemma}
\label{edgecurve}
Let $\ell$ be a cyclic length metric on $\partial C(n,4)$.  For all $v\in V$, $K_v$ is constant.
\end{lemma}

\begin{proof}
We will provide the proof for $n=2m+3$.  The case of $n=2m+2$ follows similarly.

First we will show that all edges with equal length have equal edge curvature, $K_e$. We will denote the $m+1$ possible edge lengths as $\{L_1, L_2, \ldots, L_{m+1}\}$ as in the proof of Lemma \ref{tetralemma}.  Suppose $i<j$.  Choose an arbitrary $e_{ij} \in E$ with length $L_p$.  Recall this corresponds to the case that  $D_{ij}=p$.  We will show that every edge with length $L_p$ is in the same number of certain types of tetrahedra.  By Proposition~\ref{eq length eq dih}, this will imply that the dihedral angles, and thus the edge curvatures, are the same.

First, let $p\neq1, 2,  m+1$.  The tetrahedra containing edge $e_{ij}$ clearly contain the vertices $v_i$ and $v_j$.  By Corollary~\ref{edgecor}, the other two vertices can be combinations of $v_{i+1}$ or $v_{i-1}$ and $v_{j+1}$ or $v_{j-1}$.   Simply by observation, one can determine what type of tetrahedra corresponds to each vertex choice.  For example, both the vertex sets $v_{i-1}, v_{i}, v_{j-1}, v_{j}$ and $v_{i+1}, v_{i}, v_{j+1}, v_{j}$ yield a type $T_{p-1}$ tetrahedron.  The vertex set $v_{i-1}, v_{i}, v_{j+1}, v_{j}$ yields a type $T_p$ tetrahedron, while the set $v_{i+1}, v_{i}, v_{j-1}, v_{j}$ gives a type $T_{p-2}$ tetrahedron.

Secondly, let $p=m+1$.   Again, we simply list the tetrahedra containing $e_{ij}$ and note their types.  The vertex set $v_{i+1}, v_{i}, v_{j-1}, v_{j}$ generates  a type $T_{m-1}$ tetrahedron.   However, all of the sets $v_{i-1}, v_{i}, v_{j-1}, v_{j}$; $v_{i+1}, v_{i}, v_{j+1}, v_{j}$; and $v_{i-1}, v_{i}, v_{j+1}, v_{j}$ yield type $T_m^o$ tetrahedra.

For the third case, let $p=2$.  In this case, notice that only one of $v_{i+1}$ and $v_{j-1}$ can be vertices of a tetrahedron, unless $i=1$ and $j=2m+2$.  For the former case, the vertex sets $v_{i-1}, v_{i}, v_{j-1}, v_{j}$ and $v_{i+1}, v_{i}, v_{j+1}, v_{j}$ yield tetrahedra of type $T_1$, whereas $v_{i-1}, v_{i}, v_{j+1}, v_{j}$ gives a type $T_2$ tetrahedron.  In the latter case, the analysis is the same, but notice that only one of $v_{i-1}$ and $v_{j+1}$ can be vertices.

Finally, let $p=1$.   Since there are $2m+3$ edges in $\mathcal{C}$, Corollary~\ref{edgecor} implies that there are $2m$ tetrahedra containing both $e_{ij}$ and an edge that is not local to $e_{ij}$ in $\mathcal{C}$.  Notice also that the tetrahedron with vertices $v_{i-1}, v_i, v_j, v_{j+1}$ contains $e_{ij}$ and is not one of the $2m$ tetrahedra previously mentioned.  Thus, in this case, $e_{ij}$ is contained in $2m+1$ tetrahedra.  The vertex set  $v_{i-1}, v_i, v_j, v_{j+1}$ yields a tetrahedron of type $T_1$.   Additionally, one can see that $e_{ij}$ is in two of each type of $T_k$ for $1\leq k <m$ and two of type $T_m^o$.

Therefore the edge curvature of any given edge depends only on the length of that edge, so edges with equal length have equal edge curvature.  By Remark~\ref{cl is vt}, the metric is vertex transitive,  so the sets of edge lengths of the edges incident to each vertex are  the same. Together, these facts imply that for all  $v\in V$, $K_v=\frac{1}{2}\sum_{e>v} K_e$ is constant.
\end{proof}

\begin{corollary}
\label{vertexcurve}
Let $\ell$ be a cyclic length metric on $\partial C(n,4)$.  Then $\sum_e{K_e} = nK_v$.
\end{corollary}
\begin{proof}
This result follows directly from Lemma \ref{edgecurve}. 
\end{proof}

\begin{lemma}
\label{lvlem}
Let $\ell$ be a cyclic length metric on $\partial C(n,4)$.  For all $v \in V$, $L_v=\frac{1}{n}\mathcal{L}$.\end{lemma}
\begin{proof}
Recall $L_v=\frac{1}{2}\sum_{e>v}\ell_e$, as in Definition~\ref{csc def}.  Since a cyclic length metric is vertex transitive, $L_v$ is clearly constant for all $v \in V$.    Additionally, since $\sum_{v \in V}L_v=\mathcal{L}$, we see that $L_v=\frac{1}{n}\mathcal{L}$ as required.
\end{proof}

\begin{lemma}
\label{vvlem}
Let $\ell$ be a cyclic length metric on $\partial C(n,4)$.  For all $v\in V$, $V_v=\frac{3}{n}\mathcal{V}$.
\end{lemma}
\begin{proof}
We begin the proof by letting $n=2m+3$.  Recall the definition of $V_v$ as given in Definition~\ref{csc def}.  For any vertex $v_i \in V$, $V_{v_i} = \frac{1}{3}\sum_{t>v}\sum_{(j,k,l)<t}{h_{ijk,l}A_{ijk}}$. Note that this sum is over all tetrahedra containing $v_i$.  By Lemma~\ref{fourdis}.1, each vertex is contained in exactly $4m$ tetrahedra.  More precisely, each vertex is contained in four of each type $T_p$ tetrahedra for $1\leq p< m$ and four tetrahedra of type $T_m^o$. Thus we can split up $V_{v_i}$ into $m$ partial sums, one  for each type of tetrahedron. Note that every tetrahedron of type $T_p$ is geometrically identical, so the sum $\sum_{(j,k,l)<t}{h_{ijk,l}A_{ijk}}$ only varies in regards to the choice of vertices.  

In the tetrahedra of type $T_p$, a given vertex $v_i$  has each of the sets of possible incident edges appear one time.  In other words, the vertex $v_i$ ``plays the role'' of each of the vertices $i, j, k, l$ in Figure~\ref{T_p tetra} or Figure~\ref{T_m^o tetra}.  There are four possible sets of incident edges.  Thus for all $v_i \in V$, the partial sums of $V_{v_i}$ with respect to the tetrahedra of type $T_k$ are equal.  Since this holds for every type $T_k$ for $1\leq k \leq m,$ the total sums $V_{v_i}$ are equal for all $v_i\in V$. Thus,  $V_v$ is constant, so ${\sum_v V_v} = nV_v$ for all $v\in V$.  Recall that $\sum_{v \in V}V_v=3\mathcal{V}$, so $V_v=\frac{3}{n}\mathcal{V}$ as required.. 

The proof for $n=2m+2$ follows in the same way.  The only difference is that a vertex $v_i$ is contained in only two type $T_m^e$ tetrahedra.  However, in this case, notice that all of the vertices have the same sets of incident edge lengths (see Figure~\ref{T_m^e tetra}), so again the sum $V_v$ equal to the same constant for all $v$.
\end{proof}

We now prove the main theorem of this section; namely that cyclic length metrics on $\partial C(n,4)$ have constant scalar curvature.
\begin{theorem}
Let $\ell$ be a cyclic length metric on $\partial C(n,4)$.  This metric is both $\mathcal{LCSC}$ and $\mathcal{VCSC}$; i.e. it satisfies both Eq.~\ref{Lcsc eqn} and Eq.~\ref{Vcsc eqn}.
\end{theorem}
\begin{proof}
We will begin by showing this metric is $\mathcal{LCSC}$. By Corollary~\ref{vertexcurve} and Lemma~\ref{lvlem}:
\[
\lambda_L = \frac{\sum_e K_e}{\sum_e l_e} = \frac {n K_v}{\frac{n}L_v} =\frac { K_v}{L_v} \]
for all $v\in V$. Thus, for all $v\in V$, $K_v = \lambda_\mathcal{L}L_v$.  Therefore $\ell$ is $\mathcal{LCSC}$.

We will now show that this metric is $\mathcal{VCSC}$.  Applying Corollary~\ref{vertexcurve} and Lemma~\ref{vvlem}  yields
\[\lambda_\mathcal{V} = \frac{\sum_e K_e}{3\mathcal{V}} = \frac{nK_v}{nV_v} = \frac{K_v}{V_v}\]
for all $v\in V$. Thus, for all $v\in V$, $K_v = \lambda_\mathcal{V}V_v$. Therefore, $\ell$ is $\mathcal{VCSC}$.\end{proof}

\section{Open Questions}
As mentioned in Definition~\ref{csc def}, $\mathcal{LCSC}$ and $\mathcal{VCSC}$ metrics occur at critical points of the $\mathcal{LEHR}$ and $\mathcal{VEHR}$ functionals within a conformal class.  We have shown that cyclic length metrics on $\partial C(n,4)$ are both $\mathcal{LCSC}$ and $\mathcal{VCSC}$, and we would like to know with what frequency these metrics occur; i.e. in which conformal classes  these cyclic length metrics appear.  One may hope that they occur in every conformal class.  However  the next proposition shows that this is not the case.

\begin{proposition}
\label{conform}
On $\partial C(n,4)$, there exists a conformal class that does not admit a cyclic length metric.
\end{proposition}
\begin{proof}
Recall from Definition~\ref{conf class def} that a conformal class is equivalent to a choice of $L_{ij}$ for all $e_{ij}\in E$.  Then the edge lengths within this conformal class are given by $\ell_{ij}=\exp\left[\frac{1}{2}(f_i+f_{j})\right]L_{ij}$ for $e_{ij}$.

We begin by supposing that all conformal classes on $\partial C(n,4)$ admit a cyclic length metric, $\ell$.  Choose a vertex $v_0$ and label the remaining vertices $v_1, v_2, \ldots, v_{n-1}$ proceeding clockwise around $\mathcal{C}$.  Recall that in a cyclic length metric 
\begin{equation}
\label{eq el 1}
\ell_{01}=\ell_{12}=\ell_{23}=\ell_{34}, \mbox{ and}
\end{equation}
\begin{equation}
\label{eq el 2}
\ell_{02}=\ell_{13}=\ell_{24}.
\end{equation}
Now consider the conformal class such that $L_{ij}$=1 for all $e_{ij}\in E$, except for $L_{12}=a\neq 1$.

Equation~\ref{eq el 1} implies both that $e^{\frac{1}{2}f_0} = a e^{\frac{1}{2}f_2}$ and $e^{\frac{1}{2}f_3} = a e^{\frac{1}{2}f_1}$.  Substituting this into the first equality in Equation~\ref{eq el 2} shows that $f_1=f_2$.  The equality $\ell_{12}=\ell_{34}$ then yields
\begin{eqnarray*}
a e^{\frac{1}{2}f_1} e^{\frac{1}{2}f_2} &=& e^{\frac{1}{2}f_3} e^{\frac{1}{2}f_4}\\
a (e^{\frac{1}{2}f_1})^2 &=& a e^{\frac{1}{2}f_1} e^{\frac{1}{2}f_4}\\
f_1 &=& f_4.
\end{eqnarray*}
But this implies that $l_{24}=e^{f_1}$ and $l_{13}=a e^{f_1}$.  Equation~\ref{eq el 2} then gives $a e^{f_1} = e^{f_1}$, which contradicts our choice of $a$.   Therefore this conformal class does not admit a cyclic length metric.
\end{proof}

One would like to know if there is some sort of generalization of a cyclic length metric that has constant scalar curvature and that occurs in all conformal classes.   Additionally, one might do an analysis similar to that in \cite{CGY} of the Hessians of the $\mathcal{LEHR}$ and $\mathcal{VEHR}$ at a cyclic length metric to determine the local behavior at these critical points.


\begin{thebibliography}{99}
\bibitem {BSu}\textbf{Bobenko, A.I. and Suris, Y.B.}  \emph{Discrete differential
geometry. Integrable structure.} Graduate Studies in Mathematics, 98. American
Mathematical Society, Providence, RI, 2008. xxiv+404 pp.

\bibitem{Ch} \textbf{Champion, D.}  Mobius structures, Einstein metrics, and discrete conformal variations on piecewise flat two and three dimensional manifolds, Ph.D. thesis, University of Arizona, 2010.
\bibitem{CGY} \textbf{Champion, D., Glickenstein, D., and Young, A.}  Regge's Einstein-Hilbert functional on the double tetrahedron.  To Appear in \emph{Differential Geometry and its Applications}.  Preprint at arXiv:1007.0048v1 [math.DG].
\bibitem {Che}\textbf{Cheeger, J. and Gromoll, D}. The splitting theorem for
manifolds of nonnegative Ricci curvature. \emph{J. Differential Geometry} 6
(1971/72), 119--128.
\bibitem {DHM}\textbf{Desbrun, M.,  Hirani, A. N., and Marsden, J. E.} Discrete exterior
calculus for variational problems in computer vision and graphics,
\emph{Proceedings of the 42nd IEEE Conference on Decision and Control (CDC)},
5 (2003) 4902--4907.

\bibitem {DP}\textbf{Desbrun, M.  and Polthier, K.} Foreword: Discrete differential
geometry. Comput. Aided Geom. Design 24 (2007), no. 8-9, 427.
\bibitem {Gl}\textbf{Glickenstein, D.} Discrete conformal variations and scalar
curvature on piecewise flat two and three dimensional manifolds. Preprint at
arXiv:0906.1560v1 [math.DG].
\bibitem{Gr}\textbf{Grunbaum, B.} \emph{Convex Polytopes.} Springer-Verlag New York, Inc., 2003.
\bibitem {Ham}\textbf{Hamber, H. W. } \emph{Quantum gravitation: The Feynman path
integral approach}. Springer, Berlin, 2009, 342 pp.
\bibitem{KL}\textbf{K\"{u}hnel, W. and Lassman, G. } Neighborly combinatorial 3-manifolds with dihedral automorphism group.  \emph{Israel Journal of Mathematics}, Vol 52, Nos. 1-2, 1985.
\bibitem{LP}\textbf{Lee, J. and Parker, T.}  The Yamabe Problem.  \emph{Bull. Amer. Math. Soc.}, Vol. 17, No. 1, 1987.
\bibitem {Luo1}\textbf{Luo, F.} Combinatorial Yamabe flow on surfaces. \emph{Commun. Contemp.
Math.} 6 (2004), no. 5, 765--780.
\bibitem {MDSB}\textbf{Meyer, M., Desbrun, M. ,  Schr\"{o}der, P., and Barr, A. H.} Discrete
differential-geometry operators for triangulated 2-manifolds.
\emph{Visualization and mathematics III}, 35--57, Math. Vis., Springer,
Berlin, 2003.


\bibitem {Re}\textbf{Regge, T.} General relativity without coordinates, \emph{Nuovo Cimento}
(10) 19 (1961), 558--571.

\bibitem {RW}\textbf{Ro\v{c}ek, M. and Williams, R. M.} The quantization of Regge
calculus. Z. Phys. C 21 (1984), no. 4, 371--381.
\end{thebibliography}
\end{document}